\documentclass{amsart}
\usepackage{blkarray}
\usepackage{amsmath,amssymb}
\usepackage{amsthm}
\usepackage{amsfonts}
\usepackage{dsfont}
\usepackage[all]{xy}
\usepackage{supertabular}
\usepackage{longtable}
\usepackage{graphicx}
\DeclareGraphicsExtensions{.pdf,.jpg,.png}
\usepackage{t1enc}
\usepackage[hmarginratio=1:1,top=32mm,columnsep=20pt]{geometry} 

\newtheorem{theorem}{Theorem}[section]

\newtheorem{proposition}[theorem]{Proposition}

\theoremstyle{definition}
\newtheorem{definition}[theorem]{Definition}

\theoremstyle{remark}
\newtheorem{remark}[theorem]{Remark}
\numberwithin{equation}{section}

\newcommand{\K}{\mathbb{K}}

\title[]{Classification of $3$-dimensional Bihom-Associative and Bihom-Bialgebras.}
\author{Ahmed Zahari}
\address{Universit\'{e} de Haute Alsace,  Laboratoire de Math\'{e}matiques, Informatique et Applications,
4, rue des Fr\`eres Lumi\`ere F-68093 Mulhouse, France}
\email{zaharymaths@gmail.com}


\begin{document}

\begin{abstract}
   The purpose of this paper is to study the structure and the algebraic varieties of BiHom-associative algebras. We provide  a classification of $n$-dimensional BiHom-associative  and BiHom-bialgebras and 
   BiHom Hopf algebras for  $n\leq 3.$ 
 \end{abstract}

\maketitle
\section*{Introduction}\label{sec:intro} 

The first motivation to study nonassociative BiHom-algebras comes from quasi-deformations of Lie algebras of vector fields, in particular $q$-deformations of Witt and Virasoro algebras. It was observed in the pioneering works, mainly by physicists,  that in these examples a twisted Jacobi identity holds. Motivated by these examples and their generalization on the one hand, and the desire to be able to treat within the same framework such well-known generalizations of Lie algebras as the color and Lie superalgebras on the other hand, quasi-Lie algebras and subclasses of quasi-hom-Lie algebras and hom-Lie algebras were introduced by Hartwig, Larsson and Silvestrov in \cite{HLS,LS}.
The BiHom-associative algebras play the role of associative algebras in the BiHom-Lie setting. Usual functors between the categories of Lie algebras and associative algebras were extended to Hom-setting, see \cite{Yau-Env2008} for the construction of the enveloping algebra of a BiHom-Lie algebra.

A  Hom-associative algebra $(A, \mu, \alpha, \beta)$  consists of a vector space, a multiplication and two linear endomorphism. It may be viewed as a deformation of an associative algebras, in which the associativity condition is twisted by the linear maps $\alpha$ and $\beta$ in a certain way such that when $\alpha=id$ and $\beta=id$, the BiHom-associative algebras degenerate to exactly  associative algebras. 
In this paper we aim to study the structure of BiHom-associative algebras.
 Let $A$ be an $n$-dimensional $\mathbb{K}$-linear space and  $\left\{e_1, e_2, \cdots, e_n\right\}$ be a basis of $A$. A BiHom-algebra structure on $A$ with product $\mu$ is determined by $n^3$ structure constants $\mathcal{C}_{ij}^k$, were $\mu(e_i, e_j)=\sum^n_{k=1}\mathcal{C}_{ij}^ke_k$ and by $\alpha$ and $\beta$ which is given by ${2n^2}$ structure constants $a_{ij}$ and $b_{ij}$, where 
$\alpha(e_i)=\sum_{j=1}^na_{ji}e_j$ and $\beta(e_j)=\sum_{k=1}^nb_{kj}e_k$. Requiring the algebra structure to be BiHom-associative and unital  gives rise to sub-variety $\mathcal{H}_n$ (resp. $\mathcal{H}u_n$) of $k^{n^3+2n^2}$. Base changes in $A$ result in the natural transport of the structure action of $GL_n(k)$ on $\mathcal{H}_n$. Thus the isomorphism classes of 
$n$-dimensional BiHom-algebras are in one-to-one correspondence with the orbits of the action of $GL_n(k)$ on $\mathcal{H}_n$ 
(rep. $\mathcal{H}u_n$). \\
Furthermore, we shall consider the class of BiHom-bialgebras which are
Hom-associative algebras equipped with a compatible BiHomCoalgebra structure,
in particular BiHomHopf algebras where the additional structures like
the comultiplication $\Delta$ and the counit $\epsilon$ can be expressed in a base in a similar way
as the multiplication above. We shall also give a classification
of these algebras up to isomorphism in low dimension $n\leq 3$.

The paper is organized as follows. In the first section we give the  basics about BiHom-associative algebras and provide some new properties. Moreover, we discuss unital BiHom-associative algebras. In Section $2$ is dedicated to  describe algebraic varieties of BiHom-associative algebras and provide classifications, up to isomorphism, of $2$-dimensional (resp. $3$-dimensional) BiHom-associative algebras. Moreover, in Section 3
we shall recall the definitions of BiHom-bialgebras and BiHomHopf algebras
and present the classification up to dimension $3$.

\vspace{0.5cm}

\noindent \textbf{Acknowledgements}. I'd like to thank my thesis
supervisor Abdenacer Makhlouf and Martin Bordomann for giving me the problem, for
many fruitful discussions and for their constant support.

\section{Structure of BiHom-associative algebras}\label{sec:prel}
Let $\mathbb{K}$ be an algebraically closed field  of characteristic $0$, $A$ be a linear space over $\mathbb{K}$. We refer to a Hom-algebra by a 
$4$-tuple $(A,\mu,\alpha, \beta)$, where $\mu:A\times A\rightarrow A$ is a bilinear map (multiplication) and   $\alpha$ and $\beta$ be two homomorphisms of $A$ (twist map). 

\subsection{Definitions}
\begin{definition}\cite{MS}.
A BiHom-associative algebra is a $4$-tuple $(A, \mu, \alpha, \beta)$ consisting of a linear space $A$, a bilinear map $\mu : A\times A\rightarrow A$ and 
two linear space homomorphism $\alpha : A\rightarrow A$ and $\beta : A\rightarrow A$ satisfying 
\begin{equation}\label{Homassoc}
\mu(\alpha(x), \mu(y,z))=\mu(\mu(x,y), \beta(z)).
\end{equation}
\begin{equation}
\alpha(\mu(x,y))= \mu(\alpha(x),\alpha(y))\quad and \quad \beta(\mu(x,y))= \mu(\beta(x),\beta(y)).
\end{equation}
\begin{equation}
\alpha\circ\beta=\beta\circ\alpha.
\end{equation}
Usually such BiHom-associative algebras are called multiplicative. Since we are dealing only with multiplicative  BiHom-associative algebras, we shall call them BiHom-associative algebras for simplicity. We denote the set of all BiHom-associative algebras by $\mathcal{H}$. 
In the language of Hopf algebras, the multiplication of a BiHom-associative algebra over $A$ consists of a
linear map $\mu : A\otimes A\rightarrow A$, and Condition \eqref{Homassoc} can be written as
\begin{equation}
\mu(\alpha(x)\otimes\mu(y\otimes z))=\mu(\mu(x\otimes y)\otimes\beta(z)).
 \end{equation}
\end{definition}

\begin{definition}
Let $(A, \mu_A, \alpha_A, \beta_A)$ and $(B, \mu_B, \alpha_B, \beta_B)$ be two BiHom-associative algebras. A linear map $\varphi : A\rightarrow B$ is called a BiHom-associative algebras morphism if
\begin{equation}\label{d2} 
\varphi\circ\mu_A=\mu_B(\varphi\otimes\varphi), \quad\alpha_B\circ\varphi=\varphi\circ\alpha_A\quad and\quad\beta_B\circ\varphi=\varphi\circ\beta_A.
\end{equation} 
\end{definition}
In particular, BiHom-associative algebras $(A,\mu_A,\alpha_A, \beta_A)$ and $(B, \mu_B,\alpha_B, \beta_B)$ are isomorphic if $\varphi$ is also bijective.
\begin{definition}\label{d}
 A unital BiHom-associative algebra $(A, \mu, \alpha, \beta,u)$ is called unital if there exists an element $u_A\in A$ (called a unit) such that
 $\alpha(u_A)=u_A, \quad \beta(u_A)=u_A$ and $au_A=\alpha(a)\quad$ and $\quad u_Aa=\beta(a)$, $\quad \forall a\in A$. 
 \end{definition} 
A morphism of unital BiHom-associative algebras $\phi : A\longrightarrow B$ is called unital if $\phi(u_A)=u_B.$

\subsection{Structure of BiHom-associative algebras}\  \\
We state in this section some properties on the structure of BiHom-associative algebras which are not necessarily multiplicative.
\begin{proposition}[\cite{Yau}]
Let $(A, \mu,\alpha, \beta)$ be a BiHom-associative algebra and $\gamma : A\rightarrow A$ be a BiHom-associative algebra morphism. 
Then $(A,\gamma\mu,\gamma\alpha,\gamma\beta)$ is a BiHom-associative algebra. 
\end{proposition}

\begin{definition}\label{Defs1}
Let $(A, \mu, \alpha,\beta)$ be a BiHom-associative algebra. If there is an associative algebra $(A, \mu')$ such that 
$\mu(x,y)=(\alpha\otimes\beta)\mu'(x,y),\, \forall x,y\in A$, we say that $(A, \mu, \alpha,\beta)$ is of associative type and $(A,\mu')$ is its compatible associative algebra or the untwist of $(A, \mu, \alpha,\beta)$.  
\end{definition}

\begin{proposition}\label{p1}
Let $(A, \mu, \alpha, \beta)$ be an $n$-dimensional BiHom-associative algebra and $\phi : A\rightarrow A$ be an invertible linear map. Then there is an
isomorphism with an n-dimensional BiHom-associative algebra  $(A, \mu', \phi\alpha\phi^{-1}, \phi\beta\phi^{-1})$ where 
$\mu'=\phi\circ\mu\circ(\phi^{-1}\otimes\phi^{-1}$). Furthermore, if $\left\{C^k_{ij}\right\}$ are the structure constants of $\mu$ with
respect to the basis $\left\{e_1,\dots,e_n\right\}$,  then $\mu'$ has the same structure constants with respect to the basis 
$\left\{\phi(e_1),\dots,\phi(e_n)\right\}$.
\end{proposition}

\begin{proof}
We prove for any invertible linear map $\phi : A\rightarrow A, \, (A, \mu', \phi\alpha\phi^{-1}, \phi\beta\phi^{-1})$ is a BiHom-associative algebra.
$$\begin{array}{ll}
&\mu'(\mu'(x, y),\phi\beta\phi^{-1}(z))
=\phi\mu(\phi^{-1}\otimes\phi^{-1})(\phi\mu(\phi^{-1}\otimes\phi^{-1})(x,y),\phi\beta\phi^{-1}(z))\\
&=\phi\mu(\mu(\phi^{-1}(x),\phi^{-1}(y)),\beta\phi^{-1}(z))
=\phi\mu(\alpha\phi^{-1}(x),\mu(\phi^{-1}(y),\phi^{-1}(z)))\\
&=\phi\mu(\phi^{-1}\otimes\phi^{-1})(\phi\otimes\phi)(\alpha\phi^{-1}(x),\mu(\phi^{-1}\otimes\phi^{-1})(y,z)))\\
&=\phi\mu(\phi^{-1}\otimes\phi^{-1})(\phi\alpha\phi^{1}(x),\phi\mu(\phi^{-1}\otimes\phi^{-1})(y,z)))
=\mu'(\phi\alpha\phi^{-1}(x),\mu'(y,z)).
\end{array}$$
So $(A, \mu', \phi\alpha\phi^{-1}, \phi\beta\phi^{-1})$ is a BiHom-associative algebra.\\
It is also multiplicative. Indeed, for $\alpha$  
$$\begin{array}{ll}
&\phi\alpha\phi^{-1}\mu'(x,y)
=\phi\alpha\phi^{-1}\phi\mu(\phi^{-1}\otimes\phi^{-1})(x,y)
=\phi\alpha\mu(\phi^{-1}\otimes\phi^{-1})(x,y)\\
&=\phi\mu(\alpha\phi^{-1}(x),\alpha\phi^{-1}(y))
=\phi\mu(\phi^{-1}\otimes\phi^{-1})(\phi\otimes\phi)(\alpha\phi^{-1}(x),\alpha\phi^{-1}(y))
=\mu'(\phi\alpha\phi^{-1}(x),\phi\alpha\phi^{-1}(y)).
\end{array}$$
We have also for $\beta$
$$\begin{array}{ll}
&\phi\beta\phi^{-1}\mu'(x,y)
=\phi\beta\phi^{-1}\phi\mu(\phi^{-1}\otimes\phi^{-1})(x,y)
=\phi\beta\mu(\phi^{-1}\otimes\phi^{-1})(x,y)\\
&=\phi\mu(\beta\phi^{-1}(x),\beta\phi^{-1}(y))
=\phi\mu(\phi^{-1}\otimes\phi^{-1})(\phi\otimes\phi)(\beta\phi^{-1}(x),\beta\phi^{-1}(y))
=\mu'(\phi\beta\phi^{-1}(x),\phi\beta\phi^{-1}(y)).
\end{array}$$

Therefore $\phi : (A,\mu,\alpha,\beta)\rightarrow(A,\mu',\phi\alpha\phi^{-1}, \phi\beta\phi^{-1})$ is a BiHom-associative  algebras morphism, since\\ 
$\phi\circ\mu=\phi\circ\mu\circ(\phi^{-1}\otimes\phi^{-1})\circ(\phi\otimes\phi)=\mu'\circ(\phi\otimes\phi)$ and 
$(\phi\alpha\phi^{-1})\circ\phi=\phi\circ\alpha$ and $(\phi\beta\phi^{-1})\circ\phi=\phi\circ\beta.$ \\
It is easy to see that 
$\left\{\phi(e_i), \cdots, \phi(e_n)\right\}$ is a basis of $A$. For $i,j=1,\cdots,n$, we have 

$\begin{array}{ll}
\mu_2(\phi(e_i),\phi(e_j))
&=\phi\mu_1(\phi^{-1}(e_i), \phi^{-1}(e_j))=\phi\mu(e_i, e_j)=\sum_{k=1}^n\mathcal{C}^k_{ij}\phi(e_k).
\end{array}$ 
\end{proof}

\begin{remark}
A BiHom-associative algebra $(A,\mu,\alpha, \beta)$ is isomorphic to an associative algebra if and only if $\alpha=\beta=id$. Indeed,
$\phi\circ\alpha\circ\phi^{-1}=\phi\circ\beta\circ\phi^{-1}=id$ is equivalent to  $\alpha=\beta=id$.
\end{remark}

\begin{remark}
 Proposition \ref{p1} is useful to make a classification of BiHom-associative algebras. Indeed, we have to consider the class of morphisms which are conjugate. Representations of these classes are given by Jordan forms of the matrix. 
 Any $n\times n$ matrix over $\mathbb{K}$ is equivalent up to basis change  to  Jordan's canonical form, then we choose $\phi$ such that the matrix of $\phi\alpha\phi^{-1}=\gamma$ and $\lambda=\phi\beta\phi^{-1}$, where $\gamma$ and $\lambda$ are Jordan canonical forms. \\
Hence, to obtain the classification, we consider only Jordan forms for the structure map of Hom-associative algebras.
\end{remark}

\begin{proposition}
Let $(A, \mu, \alpha, \beta)$ be a Hom-associative algebra over $\mathbb{K}$. Let $(A, \mu', \phi\alpha\phi^{-1}, \phi\beta\phi^{-1})$ be 
its isomorphic Hom-associative algebra described in Proposition \ref{p1}. If $\psi$ is an automorphism of $(A, \mu, \alpha, \beta)$, then
$\phi\psi\phi^{-1}$ is an automorphism of $(A, \mu, \phi\alpha\phi^{-1}, \phi\beta\phi^{-1})$. 
\end{proposition}

\begin{proof}
Note that $\gamma=\phi\alpha\phi^{-1}$. We have 
$$\phi\psi\phi^{-1}\gamma=\phi\psi\phi^{-1}\phi\alpha\phi^{-1}=\phi\psi\alpha\phi^{-1}=\phi\alpha\psi\phi^{-1}=\phi\alpha\phi^{-1}\phi\psi\phi^{-1}=\gamma\phi\psi\phi^{-1}$$.
For $\beta$, we pose $\lambda=\phi\beta\phi^{-1}$, we have\\
$$\phi\psi\phi^{-1}\lambda=\phi\psi\phi^{-1}\phi\beta\phi^{-1}=\phi\psi\beta\phi^{-1}=\phi\beta\psi\phi^{-1}=\phi\beta\phi^{-1}\phi\psi\phi^{-1}=\gamma\phi\psi\phi^{-1}.$$ 
For any $x,y\in A,$
$$\begin{array}{ll}
 \phi\psi\phi^{-1}\mu'(\phi(x),\phi(y))&=\phi\psi\phi^{-1}\phi\mu(x,y)
 =\phi\psi\mu(x,y)=\phi\mu(\psi(x),\psi(y))\\
 &=\mu'(\phi\psi(x), \phi\psi(y))
 =\mu'(\phi\psi\phi^{-1}(\phi(x)), \phi\psi\phi^{-1}(\phi(y))).
 \end{array}$$
 By Definition, $\phi\psi\phi^{-1}$ is an automorphism of $(A, \mu', \phi\alpha\phi^{-1}, \phi\beta\phi^{-1})$.
 \end{proof}
The following characterization was given for Hom-Lie algebras in \cite{Sheng}.

\begin{proposition}
Given two BiHom-associative algebras $(A, \mu_A, \alpha_A, \beta_A)$ and $(B, \mu_B, \alpha_B ,\beta_B)$ over field $\mathbb{K}$, there is a 
BiHom-associative algebra $(A\oplus B, \mu_{A\oplus B}, \alpha_A+\alpha_B, \beta_A+\beta_B)$, where $\mu_{A\oplus B}$ the usual multiplication\\  
$(a+b)(a'+b')=(a,a')+ (b,b').$ 

$\mu_{A\oplus B}(., .) : (A\oplus B)\times (A\oplus B)\rightarrow (A\oplus B)$ is given by 
$$\mu_{A\oplus B}(a+b,a'+b')=(\mu_A(a,a'),\mu_B(b,b')),\, \forall\, a,a'\in A,\,\forall\, b,b'\in B,$$ and the linear map 
$(\alpha_A+\alpha_B, \beta_A+\beta_B) : A\oplus B \rightarrow A\oplus B$ is given by 
$$(\alpha_A+\alpha_B, \beta_A+\beta_B)(a,b)=((\alpha_A+\beta_A)a,(\alpha_B+\beta_B)b)\,\forall a\in A , b\in B.$$  
\end{proposition}
\begin{proof}
For any $a, a', a''\in A, \, b, b', b''\in B$, by direct computation, we get
$$\begin{array}{ll} 
&\mu_{A\oplus B}((\alpha_A+\beta_A,\alpha_B+\beta_B, )(a,b), \mu_{A\oplus B}(a'+b',a''+ b''))=\\
&=\mu_{A\oplus B}(((\alpha_A+\beta_A)a,(\alpha_B+\beta_B)b), (\mu_A(a',a''), \mu_B(b',b'')))\\
&=(\mu_A(((\alpha_A+\beta_A)a), \mu_A(a',a'')), \mu_B((\alpha_B+\beta_B)b,\mu_B(b',b'')))\\
&=(\mu_A(\mu_A(a,a'), (\alpha_A+\beta_A)a''),\mu_B(\mu_B(b,b'),(\alpha+\beta)b))\\
&=(\mu_{A\oplus B}(\mu_{A\oplus B}(a+ b,a'+b'),(\alpha_A+\beta_A, \alpha_B+\beta_B)(a'',b''))).
\end{array}$$
This ends the proof.
\end{proof}
\noindent A morphism of BiHom-associative algebras 
$\phi : (A, \mu_A, \alpha_A, \beta_A)\rightarrow (B, \mu_B, \alpha_B, \beta_B)$ is a linear map $\phi : A\rightarrow B$ such that
$\phi\circ\mu_A(a, b)=\mu_B\circ(\phi(a),\phi(b)),\forall a, b\in A,\quad \alpha_B\circ\phi=\phi\circ\alpha_A$ and $\beta_B\circ\phi=\phi\circ\beta_B$

Denote by $\xi_\phi\subset A\oplus B$, the graph of the linear map $\phi : A\rightarrow B.$
\begin{proposition}
A linear map $\phi : (A, \mu_A, \alpha_A, \beta_A)\rightarrow (B, \mu_B, \alpha_B, \beta_B)$ is a BiHom-associative algebras morphism if and only if the graph $\xi_\phi\subset A\oplus B$ is a BiHom-associative sub-algebras of $(A\oplus B, \mu_{A\oplus B}, \alpha_A+\alpha_B, \beta_A+\beta_B)$. 
\end{proposition}  
\begin{proof}
Let $\phi : (A, \mu_A, \alpha_A, \beta_A)\rightarrow (B, \mu_B, \alpha_B, \beta_B)$ be a  BiHom-associative algebra morphism. Then for any 
$a, b\in A,$ we have $$\mu_{A\oplus B}((a, \phi(a)), (b, \phi(b))=(\mu_A(a,b),\mu_B(\phi(a),\phi(b)))=(\mu_A(a,b),\phi\mu_A(a,b)).$$ 
Thus the graph $\xi_\phi$ is closed under the product $\mu_{A\oplus B}$. Furthermore, since $\alpha_B\circ\phi=\phi\circ\alpha_A$ and 
$\beta_B\circ\phi=\phi\circ\beta_A$
 we have 
$$(\alpha_A+\alpha_B, \beta_A+\beta_B)(a, \phi(a))=((\alpha_A+\beta_A)(a),(\alpha_B+\beta_B)\circ\phi(a))
=((\alpha_A+\beta_A)(a), \phi\circ(\alpha_A+\beta_A(a))),$$ which implies that 
$(\alpha_A+\alpha_B, \beta_A+\beta_B)\subset \xi_\phi.$ Thus $\xi_\phi$ is a BiHom-associative sub-algebra of
$(A\oplus B, \mu_{A\oplus B},\alpha_A+\alpha_B, \beta_A+\beta_B)$.

Conversely, if the graph $\xi_\phi\subset A\oplus B$ is a Hom-associative sub-algebra of 
$(A\oplus B, \mu_{A\oplus B}, \alpha_A+\alpha_B, \beta_A+\beta_B)$, then we have 
 $$\mu_{A\oplus B}((a,\phi(a)),(b,\phi(b))=(\mu_A(a,b),\mu_B(\phi(a),\phi(b))\in \xi_\phi,$$ which implies that 
 $\mu_B(\phi(a), \phi(b))=\phi\circ\mu_A(a,b).$ Furthermore, $(\alpha_A+\beta_A, \alpha_B+\beta_B)(\xi_\phi)\subset \xi_\phi$ yields that 
 $$(\alpha_A+\alpha_B, \beta_A+\beta_B)(a,\phi(a))=((\alpha_A+\beta_A)a, \phi\circ(\alpha_A+\beta_A)a)\in \xi_\phi,$$ which is equivalent to the condition 
$\alpha_B\circ\phi(a)=\phi\circ\alpha_A(a)$ and $\beta_B\circ\phi(a)=\phi\circ\beta_A(a).$ Therefore, $\phi$ is a  BiHom-associative algebra morphism. 
\end{proof}

 \subsection{Unital BiHom-associative algebras}\  \\
In this section we discuss unital BiHom-associative algebras.  We   denote by  $\mathcal{H}u_n$ the set of $n$-dimensional unital 
BiHom-associative algebras.
\begin{definition}
A BiHom-associative algebra $(A,\mu,\alpha, \beta)$ is called unital if there exists an element $u\in A$ such that 
$\mu(x, u)=\alpha(x)$ and $\mu(u, x)=\beta(x) $ for all $x\in A$.
\end{definition}
 
\begin{proposition}\label{p3}
Let  $(A, \mu,\alpha, \beta)$ be a BiHom-associative algebra. We set $\tilde{A}=span(A, u)$ the vector space generated by elements of 
$A$ and $u$. Assume  $\mu(x,u)=\alpha(x),\, \text{and}\, \mu(u, x)=\beta(x)\, \forall\, x\in A,\, \alpha(u)=u$ and $\beta(u)=u$. 
Then $(\tilde{A}, \mu,\alpha,\beta, u)$ is a unital
Hom-associative algebra.
\end{proposition}
\begin{proof}
It is straightforward  to check the Hom-associativity. For example 
 
$\begin{array}{ll}
\mu(\mu(x,y), \beta(u))=\mu(\mu(x, y), u)=\alpha(\mu(x,y))=\mu(\alpha(x),\alpha(y))=\mu(\alpha(x),\mu(y,u)).
\end{array}$
\end{proof}
 \begin{remark}
 
Some unital BiHom-associative cannot be obtained as an extension of a non unital BiHom-associative algebra.
\end{remark}
\begin{remark}
Let $(A, \mu, \alpha, \beta, u)$ be an $n$-dimensional unital BiHom-associative algebra and $\phi : A\rightarrow A$ be an invertible linear map
such that $\phi(u)=u$. Then it is isomorphic to a $n$-dimensional BiHom-associative algebra $(A, \mu',\phi\alpha\phi^{-1},\phi\beta\phi^{-1}, u)$
where $\mu'=\phi\circ\mu\circ(\phi^{-1}\otimes\phi^{-1}$). Moreover, if $\left\{C^k_{ij}\right\}$ are the structure constants of $\mu$ with respect to the basis $\left\{e_1,\dots,e_n\right\}$ with $e_1=u$ being the unit, then $\mu'$ has the same structure constants with respect to the basis 
$\left\{\phi(e_1),\dots,\phi(e_n)\right\}$ with $u$ the unit element.

Indeed, we use  Proposition \ref{p1} and Definition \ref{d}. The unit is conserved since 
$$
\mu'(x, e_1)=\phi\circ\mu(\phi^{-1}(x), \phi^{-1}(e_1))=\phi\circ\alpha\circ\phi^{-1}(x)\, \text{and}\, 
\mu'(e_1, x)=\phi\circ\mu(\phi^{-1}(e_1), \phi^{-1}(x))=\phi\circ\beta\circ\phi^{-1}(x)
$$
\end{remark}

\begin{proposition}
Let $(A, \mu_A, \alpha_A,\beta_A, u_A)$ and $(B, \mu_B, \alpha_B, \beta_B, u_B)$ be two unital  BiHom-associative algebras with $\phi(u_A)=u_B$. Suppose there exists a BiHom-associative algebra morphism $\phi : A\rightarrow B$. If $(A, \mu'_A, u'_A)$ is an untwist of 
$(A, \mu_A, \alpha_A,\beta_A, u_A)$ then there exists an untwist of $(B, \mu_B, \alpha_B,\beta_B, u_B)$ such that 
$\phi : (A, \mu'_A, u'_A)\rightarrow (B, \mu'_B, u'_B)$ is an algebra morphism. 
\end{proposition}
\begin{proof}
Because $\phi$ is a homomorphism from $(A, \mu_A, \alpha_A,\beta_A, u_A)$ to $(B, \mu_B, \alpha_B, \beta_B, u_B)$. Then 
$\alpha_B\phi=\phi\alpha_A,$ and $\beta_B\circ\phi=\phi\circ\beta_A$ for all $x\in A$, we have 
$\mu_B(\phi(x),\phi(u_A))=\mu_B(\phi(x),u_B)=\alpha_B\circ\phi(x)$ and $\mu_B(\phi(u_A),\phi(x))=\mu_B(u_B, \phi(x))=\beta_B\circ\phi(x)$. We have also 
$\phi\circ\mu_A(x,u_A)=\phi\circ\alpha_A(x)$ and $\phi\circ\mu_A(u_A, x)=\phi\circ\beta_A(x)$. By Proposition \ref{p3}, we can see that
$(A_B, \mu_B, u_B)$ is also an associative algebra.
Furthermore

$\begin{array}{ll}
\mu'_B(\phi(x), \phi(u_A))=\mu'_B(\phi(x), u_B)
&=\phi\circ\alpha'_A\circ\phi(x)=\phi\circ\alpha_A\circ\mu_A(x,u_A)\\
&=\alpha_B\circ\phi\circ\mu_A(x,u_A)=\alpha_B\circ\mu_B(\phi(x), u_B)\quad \text{and}
\end{array}$

$\begin{array}{ll}
\mu'_B(\phi(u_A), \phi(x))=\mu'_B(u_B, \phi(x))
&=\phi\circ\beta'_A\circ\phi(x)=\phi\circ\alpha_A\circ\mu_A(u_A, x)\\
&=\beta_B\circ\phi\circ\mu_A(u_A, x)=\beta_B\circ\mu_B(u_B, \phi(x)).
\end{array}$

\end{proof}

\section{Algebraic varieties of BiHom-associative algebras and Classifications} 
In this section, we deal with Algebraic varieties of BiHom-associative algebras with a fixed dimension. A BiHom-associative algebra is identified with its structures constants with respect to a fixed basis. Their set corresponds to an algebraic variety where the ideal is generated by polynomials corresponding to the BiHom-associativity condition.

\subsection{Algebraic varieties $\mathcal{H}_n$ and their action of linear group}. \ \\
Let $A$ be a $n$-dimensional $\K$-linear space and $\left\{e_1, \cdots,e_n\right\}$ be a basis of $A$. A BiHom-algebra structure on $A$ with product $\mu$ and a structure map $\alpha$ and $\beta$ is determined by $n^3$ structure constants $\mathcal{C}_{ij}^k$ where  
$\mu(e_i, e_j)=\sum_{k=1}^n\mathcal{C}_{ij}^ke_k$ and by $2n^2$ structure constants $a_{ji}$ and $b_{kj}$, where 
$\alpha(e_i)=\sum_{j=1}^na_{ji}e_j$ and $\beta(e_j)=\sum_{k=1}^nb_{kj}e_k.$ 
If we require this algebra structure to be BiHom-associative, then this limits the set of structure constants 
($\mathcal{C}_{ij}^k, a_{ij}, b_{jk})$ to a cubic sub-variety of the affine algebraic variety $\mathbb{K}^{n^3+2n^2}$ defined by the following polynomial equations system : 
\begin{equation}\label{s1}
\left\{\begin{array}{c}  
\sum_{l=1}^n\sum_{m=1}^na_{il}\mathcal{C}_{jk}^m\mathcal{C}^s_{lm}-b_{mk}\mathcal{C}_{ij}^l\mathcal{C}_{lm}^s=0,\quad \forall\,i,j,k,s\in\left\{1,\dots,n\right\},\\
\sum_{p=1}^na_{sp}\mathcal{C}_{ij}^p-\sum_{p=1}^n\sum_{q=1}^na_{pi}a_{qj}\mathcal{C}_{pq}^s=0,\quad \forall\, i,j, s\in\left\{1,\dots,n\right\},\\
\sum_{p=1}^nb_{rp}\mathcal{C}_{jk}^p-\sum_{p=1}^n\sum_{q=1}^nb_{pj}b_{qk}\mathcal{C}_{pq}^s=0,\quad\forall\, j, k, r\in\left\{1,\dots,n\right\},\\
\sum_{j=1}^nb_{ji}a_{kj}-\sum_{j=1}^na_{ji}b_{kj}b_{qk}\quad==0,\quad\forall\, j,k\in\left\{1,\dots,n\right\}.
\end{array}\right.
\end{equation}
Moreover if $\mu$ is commutative, 
we have $\displaystyle\mathcal{C}_{ij}^k=\displaystyle\mathcal{C}_{ji}^k\quad i, j, k=1,\cdots,n.$ 
The first set of equation correspond to the BiHom-associative condition 
$\mu(\alpha(e_i),\mu(e_j,e_k))=\mu(\mu(e_i,e_j), \beta(e_k))$ and the second set to multiplicativity condition 
$\alpha\circ\mu(e_i,e_j)=\mu(\alpha(e_i), \alpha(e_j))$ and $\beta\circ\mu(e_j,e_k)=\mu(\beta(e_j), \beta(e_k))$.  
We denote by $\mathcal{H}_n$ the set of all $n$-dimensional multiplicative Hom-associative algebras.  

The group $GL_n(\mathbb{K})$ acts on the algebraic varieties of  BiHom-structures by the so-called transport of structure action defined as follows. Let $A=(A,\mu,\alpha,\beta )$ be a $n$-dimensional  BiHom-associative algebra defined by multiplication $\mu$ and a linear map 
$\alpha$ and $\beta$. Given $f\in GL_n(\mathbb{K})$, the action $f\cdot A$ transports the structure, 
$$\begin{array}{lll}
\Theta : & GL_n(\mathbb{K})\times \mathcal{H}_n&\longrightarrow\mathcal{H}_n \\
                           \ &      (f, (A,\mu, \alpha, \beta))&\longmapsto(A,f^{-1}\circ\mu\circ ( f\otimes f), f\circ\alpha\circ f^{-1}, f\circ\beta\circ f^{-1})                 
\end{array}$$
defined for $x,y \in A$  by
\begin{equation}
f\cdot\mu(x, y)=f^{-1}\mu(f(x), f(y)) , 
\quad 
f\cdot\alpha(x)=f^{-1}\alpha(f(x)),\quad f\cdot\beta(x)=f^{-1}\beta(f(x)). 
\end{equation}
The conjugate class is given by  $\Theta(f, (A,\mu, \alpha, \beta))=(A,f^{-1}\circ\mu\circ ( f\otimes f), f\circ\alpha\circ f^{-1}, f\circ\beta\circ f^{-1}) )$ for $f\in GL_n(\mathbb{K}).$

The orbit of a  Hom-associative algebra $A$ of  $\mathcal{H}as_n$ is given by 
$$
\vartheta(A)=\left\{A'=f\cdot A, \, f\in GL_n(\mathbb{K})\right\}.
$$ The orbits are in \textbf{1-1 correspondence} with the isomorphism classes of $n$-dimensional BiHom-associative algebras.

The stabilizer is 
$$Stab((A,\mu, \alpha, \beta))=\left\{f\in GL_n(\mathbb{K})|(f^{-1}\circ\mu\circ( f\otimes f)=\mu \text{ and  }f\circ \alpha =\alpha \circ f
\text{ and  }f\circ \beta =\beta \circ f\right\}.$$

We characterize --in terms of structure constants-- the fact  that two BiHom-associative algebras are in the same  orbit (or isomorphic).
Let $(A, \mu_A, \alpha_A, \beta_A)$ and $(B, \mu_B, \alpha_B, \beta_B)$ be two $n$-dimensional BiHom-associative algebras. They are isomorphic if there exists
$\varphi\in GL_n(\mathbb{K})$ such that
\begin{equation}\label{c.5} 
\varphi\circ\mu_A=\mu_B(\varphi\otimes\varphi), \quad\alpha_B\circ\varphi=\varphi\circ\alpha_A\quad and\quad\beta_B\circ\varphi=\varphi\circ\beta_A.
\end{equation} 

 \begin{remark}
 Conditions (\ref{c.5}) are equivalent to
$\mu_A=\varphi^{-1}\circ\mu_B\circ\varphi\otimes\varphi$, $\alpha_A=\varphi^{-1}\circ\alpha_B\circ\varphi$ and 
$\beta_A=\varphi^{-1}\circ\beta_B\circ\varphi$. 
\end{remark}
\noindent We  set with respect to a basis $\left\{e_i\right\}_{i=1,\cdots,n}$:

$\varphi(e_i)=\sum_{p=1}^nd_{pi}e_p,\quad \alpha(e_i)=\sum_{j=1}^na_{ji}e_j,\quad
 \beta(e_i)=\sum_{k=1}^nb_{ki}e_k, i\in\left\{1,\dots,n\right\}$ 

$\mu_A(e_i, e_j)=\sum_{k=1}^n\mathcal{C}_{ij}^ke_k,\quad\mu_B(e_i, e_j)=\sum_{k=1}^n\mathcal{\tilde{C}}_{ij}^ke_k,\quad i,j
\in\left\{1,\dots,n\right\}.$
Conditions (\ref{c.5}) translate to the following
system
\begin{equation}
\left\{\begin{array}{c}  
\sum_{k=1}^n\mathcal{C}_{ij}^kd_{pk}-\sum_{k=1}^n\sum_{q=1}^nd_{ki}d_{qj}\mathcal{\tilde{C}}^s_{kq},\quad\forall\, i,j,p\in\left\{1,\dots,n\right\},\\
\sum_{p=1}^n(d_{pi}a_{qp}-a_{pi}d_{pq}),\quad\quad\forall\, i,q \in\left\{1,\dots,n\right\},\\
\sum_{p=1}^n(d_{pi}b_{qp}-b_{pi}d_{qp}),\quad \quad\forall\, i,q \in\left\{1,\dots,n\right\}.
\end{array}\right.
\end{equation}

We shall check whether the previous are isomorphic. In particular, we shall provide all $2$-dimensional BiHom-associative algebras, corresponding to solutions of the system \eqref{s1}. To this end, we use a computer algebra system.
\subsection{Classification in low dimensions}
	\begin{theorem}\label{TheoI} 
 Every $2$-dimensional multiplicative BiHom-associative algebra is isomorphic to one of the following pairwise non-isomorphic BiHom-associative algebras $(A, \ast,\alpha, \beta)$ appearing in Table 1,
 where $\ast$ is the multiplication and $\alpha$ and $\beta$ the structure maps. We set $\left\{e_1, e_2\right\}$ to be a basis of $\mathbb{K}^2$. 
\end{theorem}
\textbf{Table 1}
\begin{itemize}
\item [$\mathcal{H}^2_1$] : $e_1\ast e_1=e_2,\,e_1\ast e_2=e_2,\,e_2\ast e_1=-e_1,\,e_2\ast e_2=e_2,\,\alpha(e_1)=e_1,\,\alpha(e_2)=e_2,\,\beta(e_1)=-e_1,\,\beta(e_2)=e_2;$
\item [$\mathcal{H}^2_2$] : $e_1\ast e_1=e_2,\,e_1\ast e_2=e_1,\,e_2\ast e_1=-e_1,\, e_2\ast e_2=-e_2,\,\alpha(e_1)=-e_1,\,\alpha(e_2)=e_2,\,
\beta(e_1)=e_1,\,\beta(e_2)=e_2;$ 
\item [$\mathcal{H}^2_3$] : $e_1\ast e_1=e_1,\,e_1\ast e_2=e_2,\,e_2\ast e_1=e_2,\, e_2\ast e_2=e_2,\,\alpha(e_1)=e_1,\,\alpha(e_2)=e_2,\,
\beta(e_1)=e_1,\,\beta(e_2)=e_2;$ 
\item [$\mathcal{H}^2_4$] : $e_1\ast e_1=e_1,\,e_1\ast e_2=e_2,\,e_2\ast e_1=e_2,\, e_2\ast e_2=e_1,\,\alpha(e_1)=e_1,\,\alpha(e_2)=e_2,\,
\beta(e_1)=e_1,\,\beta(e_2)=e_2;$ 
\item [$\mathcal{H}^2_5$] : $e_1\ast e_1=e_1,\,e_1\ast e_2=e_2,\,e_2\ast e_1=e_1,\, e_2\ast e_2=e_2,\,\alpha(e_1)=e_1,\,\alpha(e_2)=e_2,\,
\beta(e_1)=e_1,\,\beta(e_2)=e_2;$ 
\item [$\mathcal{H}^2_6$] : $e_1\ast e_1=e_1,\,e_1\ast e_2=e_1,\,e_2\ast e_1=e_1,\, e_2\ast e_2=e_2,\,\alpha(e_1)=e_1,\,\alpha(e_2)=e_2,\,
\beta(e_1)=e_1,\,\beta(e_2)=e_2;$ 
\item [$\mathcal{H}^2_7$] : $e_1\ast e_2=e_1,\,e_2\ast e_1=-e_1,\, e_2\ast e_2=-e_2,\,\alpha(e_1)=-e_1,\,\alpha(e_2)=e_2,\,
\beta(e_1)=e_1,\,\beta(e_2)=e_2;$ 
\item [$\mathcal{H}^2_{8}$] : $e_1\ast e_1=-e_1,\,e_1\ast e_2=-e_2,\, e_2\ast e_2=e_2,\,\alpha(e_1)=e_1,\,\alpha(e_2)=-e_2,\,
\beta(e_1)=e_1,\,\beta(e_2)=e_2;$ 
\item [$\mathcal{H}^2_{9}$] : $e_1\ast e_1=e_1,\,e_1\ast e_2=-e_2,\,e_2\ast e_1=e_2,\,\alpha(e_1)=e_1,\,\alpha(e_2)=e_2,\,
\beta(e_1)=e_2,\,\beta(e_2)=-e_2;$ 
\item [$\mathcal{H}^2_{10}$] : $e_1\ast e_2=e_1,\,e_2\ast e_1=e_1,\,e_2\ast e_2=e_1+e_2,\,\alpha(e_1)=e_1,\,\alpha(e_2)=e_2,\,\beta(e_1)=e_1,\,\beta(e_2)=e_2;$
\item [$\mathcal{H}^2_{11}$] : $e_1\ast e_1=e_2,\,e_1\ast e_2=e_1,\,e_2\ast e_1=e_1,\,e_2\ast e_2=e_2,\,\alpha(e_1)=e_1,\,\alpha(e_2)=e_2,\,\beta(e_1)=e_1,\,\beta(e_2)=e_2;$
\item [$\mathcal{H}^2_{12}$] : $e_1\ast e_2=e_1,\quad e_2\ast e_1=e_1,\quad e_2\ast e_2=e_1,\quad\beta(e_2)=e_1;$
\item [$\mathcal{H}^2_{13}$] : $e_2\ast e_2=e_1,\quad\alpha(e_1)=e_1,\quad \alpha(e_2)=e_1+e_2,\quad\beta(e_1)=e_1,\quad \beta(e_2)=e_2.$
\end{itemize}
\begin{theorem}\label{TheoII} 
 Every $2$-dimensional unital multiplicative BiHom-associative algebra is isomorphic to one of the following pairwise non-isomorphic BiHom-associative algebras $(A, \ast,\alpha, \beta)$ appearing in Table 2,
 where $\ast$ is the multiplication and $\alpha$ and $\beta$ the structure maps. We set $\left\{e_1, e_2\right\}$ to be a basis of 
	$\mathbb{K}^2$ where $e_1$ is  the unit :  
\end{theorem}
\textbf{Table 2}
\begin{itemize}
	\item [$\mathcal{H}u^2_1$] : $e_1\ast e_1=e_1,\,e_1\ast e_2=e_2,\,e_2\ast e_1=e_2,\, e_2\ast e_2=e_1+e_2,\,\alpha(e_1)=\beta(e_1)=e_1,\,
	\alpha(e_2)=\beta(e_2)=e_2;$
	\item [$\mathcal{H}u^2_2$] : $e_1\ast e_1=e_1,\quad e_1\ast e_2=e_2,\quad e_2\ast e_1=e_2,\quad\alpha(e_1)=\beta(e_1)=e_1,\,
	\alpha(e_2)=\beta(e_2)=e_2;$
 \item [$\mathcal{H}u^2_3$] : $e_1\ast e_1=e_1,\,e_1\ast e_2=-e_2,\,e_2\ast e_1=-e_2,\,e_2\ast e_2=e_1,\,\alpha(e_1)=\beta(e_1)=e_1,\,
	\alpha(e_2)=\beta(e_2)=-e_2;$
\item [$\mathcal{H}u^2_4$] : $e_1\ast e_1=e_1,\,e_1\ast e_2=e_2,\,e_2\ast e_1=e_2,\,e_2\ast e_2=e_2,\,\alpha(e_1)=\beta(e_1)=e_1,\,
	\alpha(e_2)=\beta(e_2)=e_2.$
\end{itemize}

	\begin{theorem}
 Every $3$-dimensional BiHom-associative algebra is isomorphic to one of the following pairwise non-isomorphic BiHom-associative algebras $(A, \ast,\alpha, \beta)$ where $\ast$ is the multiplication and $\alpha$ and $\beta$ the structure maps. We set $\left\{e_1, e_2,e_3\right\}$ to be a basis of $\mathbb{K}^3$. 
\end{theorem}
\begin{itemize} 
	\item [$\mathcal{H}^3_1$] : $e_1\ast e_1=e_1,\,e_1\ast e_2=e_1,\, e_2\ast e_1=e_2,\,e_2\ast e_2=e_2,\,e_3\ast e_2=e_3,\, e_3\ast e_3=e_3,\,
	\alpha(e_1)=e_1,\, \alpha(e_2)=e_2,\\ \beta(e_1)=e_1,\quad\beta(e_2)=e_1-e_2;$ 
	\item [$\mathcal{H}^3_2$] : $e_1\ast e_1=e_1,\,e_1\ast e_2=e_1,\, e_1\ast e_3=e_3,\,e_2\ast e_1=e_2,\, e_2\ast e_2=e_2,\, e_2\ast e_3=e_3,\,
	e_3\ast e_1=e_3,\, e_3\ast e_2=e_3,\\ \alpha(e_1)=e_1,\,\alpha(e_2)=e_2,\, \alpha(e_3)=e_3,\, \beta(e_1)=e_1,\,\beta(e_2)=e_1;$ 
	\item [$\mathcal{H}^3_3$] : $e_1\ast e_1=e_1,\,\, e_1\ast e_2=e_1,\,\,e_2\ast e_2=e_2,\, e_1\ast e_3=-e_3,\quad e_3\ast e_2=e_1-e_2,\, 
	\alpha(e_1)=e_1,\,\alpha(e_2)=e_2,\\ \alpha(e_3)=-e_3,\, \beta(e_1)=e_1,\,\beta(e_2)=e_1;$
		\item [$\mathcal{H}^3_4$] : $e_1\ast e_1=e_1,\quad e_1\ast e_2=e_1-e_2,\quad e_2\ast e_1=e_2,\quad e_2\ast e_2=-e_1,\,\alpha(e_1)=e_1,\, 
		\alpha(e_2)=e_2,\, \beta(e_1)=e_1;$
	\item [$\mathcal{H}^3_5$] : $e_1\ast e_1=e_1,\quad e_1\ast e_2=e_1,\, e_2\ast e_1=e_2,\quad e_2\ast e_2=e_2,\, e_3\ast e_1=e_3,\quad 
	e_3\ast e_2=e_3,\,e_3\ast e_3=e_1-e_3,\\ \alpha(e_1)=e_1,\quad \alpha(e_2)=e_2,\quad \alpha(e_3)=e_3,\quad\beta(e_1)=e_1,\quad\beta(e_2)=e_1;$
	\item [$\mathcal{H}^3_{6}$] : $e_1\ast e_1=e_1,\,e_1\ast e_2=e_1,\,e_1\ast e_3=e_3,\,e_2\ast e_1=e_2,\,e_2\ast e_2=e_2,\,e_2\ast e_3=e_3,\, e_3\ast e_1=e_3,\,e_3\ast e_2=e_3,\\ \alpha(e_1)=e_1,\,\alpha(e_2)=e_2,\,\alpha(e_3)=e_3,\,\beta(e_1)=e_1,\,\beta(e_2)=e_1;$
\item [$\mathcal{H}^3_{7}$] : $e_1\ast e_1=e_1,\quad e_1\ast e_2=e_2,\quad e_3\ast e_2=e_3,\quad e_3\ast e_3=e_3,\quad \alpha(e_1)=e_1,
\quad\alpha(e_2)=e_2,\quad\beta(e_1)=e_1;$
\item [$\mathcal{H}^3_{8}$] : $e_1\ast e_3=e_1-e_2,\,e_2\ast e_3=e_1-e_2,\,e_3\ast e_3=e_1-e_2,\,\alpha(e_1)=e_1,\,\alpha(e_2)=e_1-e_2,\,\alpha(e_3)=e_2-e_3,\\ \beta(e_1)=e_1,\,\beta(e_2)=e_1-e_2,\,\beta(e_3)=e_2-e_3;$
\item [$\mathcal{H}^3_{9}$] : $e_2\ast e_3=-e_1,\,e_3\ast e_2=e_1,\,e_3\ast e_3=e_1,\,\alpha(e_1)=e_1,\,\alpha(e_2)=e_1+e_2,\,\alpha(e_3)=e_2+e_3,\,\, \beta(e_1)=e_1,\\ \beta(e_2)=e_1+e_2,\,\beta(e_3)=e_2+e_3;$
\item [$\mathcal{H}^3_{10}$] : $e_1\ast e_3=e_1+e_2,\,e_2\ast e_3=e_1+e_2,\,e_3\ast e_1=e_1-e_2,\,e_3\ast e_2=e_1-e_2,\,e_3\ast e_3=e_1-e_2,\,\alpha(e_1)=e_1,\\ \alpha(e_2)=e_1-e_2,\,\alpha(e_3)=e_2-e_3,\,\, \beta(e_1)=e_1,\, \beta(e_2)=e_1-e_2,\,\beta(e_3)=e_2-e_3;$
\item [$\mathcal{H}^3_{11}$] : $e_2\ast e_3=-e_1,\,e_3\ast e_2=e_1,\,e_3\ast e_3=e_1,\,\alpha(e_1)=e_1,\,\alpha(e_2)=e_1+e_2,\,\alpha(e_3)=e_1+e_2+e_3,\,\, \beta(e_1)=e_1,\\ \beta(e_2)=e_1+e_2,\,\beta(e_3)=e_1+e_2+e_3;$
\item [$\mathcal{H}^3_{12}$] : $e_1\ast e_3=e_1-e_2,\,e_2\ast e_3=e_1-e_2,\,e_3\ast e_3=e_1-e_2,\,\alpha(e_1)=e_1,\,\alpha(e_2)=e_1,\,
\alpha(e_3)=e_2,\,\, \beta(e_1)=e_1,\\ \beta(e_2)=e_1,\,\beta(e_3)=e_2;$
\item [$\mathcal{H}^3_{13}$] : $e_1\ast e_3=e_1-e_2,\,e_2\ast e_3=e_1-e_2,\,e_3\ast e_1=e_1-e_2,\,e_3\ast e_2=e_1-e_2,\,e_3\ast e_3=e_1-e_2,\,\alpha(e_1)=e_1,\,\alpha(e_2)=e_1,\\ \alpha(e_3)=e_2,\,\, \beta(e_1)=e_1,\,\beta(e_2)=e_1,\,\beta(e_3)=e_2.$
\end{itemize}

\begin{theorem}
 Every $3$-dimensional unital BiHom-associative algebra is isomorphic to one of the following pairwise non-isomorphic BiHom-associative algebras $(A, \ast,\alpha, \beta)$ where $\ast$ is the multiplication and $\alpha$ and $\beta$ the structure maps. We set $\left\{e_1, e_2,e_3\right\}$ to be a basis of $\mathbb{K}^3$. 
\end{theorem}
\begin{itemize}
	\item [$\mathcal{H}u^3_1$] : $e_1\ast e_1=e_1,\,e_1\ast e_2=e_1-e_2,\quad e_2\ast e_1=e_2,\,e_2\ast e_2=-e_1,\,e_3\ast e_3=-e_3\quad
	\alpha(e_1)=e_1,\,\alpha(e_2)=e_2,\,\beta(e_1)=e_1,\,\beta(e_2)=e_1-e_2;$ 
	\item [$\mathcal{H}u^3_2$] : $e_1\ast e_1=e_1,\,\, e_1\ast e_2=e_1,\,\, e_2\ast e_1=e_2,\,\, e_2\ast e_2=e_2,\,\,e_2\ast e_3=e_1-e_2,\,\,
	e_3\ast e_1=e_3,\,\,e_3\ast e_2=e_3,\\ e_3\ast e_3=e_1-e_2,\,\alpha(e_1)=e_1,\, \alpha(e_2)=e_2,\,\alpha(e_3)=e_3,\,\beta(e_1)=e_1,\,
	\beta(e_2)=e_1;$ 
	\item [$\mathcal{H}u^3_3$] : $e_1\ast e_1=e_1,\,e_1\ast e_2=e_1,\, e_2\ast e_1=e_2,\,e_2\ast e_2=e_2,\, e_3\ast e_1=e_3,\,e_3\ast e_2=e_3,\,
	\alpha(e_1)=e_1,\,\alpha(e_2)=e_2,\\ \alpha(e_3)=e_3,\, \beta(e_1)=e_1,\,\beta(e_2)=e_1;$  
	\item [$\mathcal{H}u^3_4$] : $e_1\ast e_1=e_1,\,e_1\ast e_2=e_1,\, e_2\ast e_1=e_2,\,e_2\ast e_2=e_2,\, e_3\ast e_3=e_3,\,
	\alpha(e_1)=e_1,\,\alpha(e_2)=e_2,\,\beta(e_1)=e_1,\\ \beta(e_2)=e_1;$
	\item [$\mathcal{H}u^3_{5}$] : $e_1\ast e_1=e_1,\,e_1\ast e_2=e_1-e_2,\, e_2\ast e_1=e_2,\,e_3\ast e_1=e_3,\,\alpha(e_1)=e_1,\,
	\alpha(e_2)=e_2,\,\alpha(e_3)=e_3,\\ \beta(e_1)=e_1,\,\beta(e_2)=e_1-e_2;$
	\item [$\mathcal{H}u^3_{6}$] : $e_1\ast e_1=e_1,\,e_1\ast e_2=e_1,\quad e_2\ast e_1=e_2,\,e_2\ast e_2=e_2,\quad e_2\ast e_3=e_1-e_2,\,
	e_3\ast e_1=e_3,\quad e_3\ast e_2=e_3,\\ \alpha(e_1)=e_1,\,\alpha(e_2)=e_2,\,\alpha(e_3)=e_3,\,\beta(e_1)=e_1,\,\beta(e_2)=e_1;$
	\item [$\mathcal{H}u^3_{7}$] : $e_1\ast e_1=e_1,\,e_1\ast e_2=e_1,\quad e_2\ast e_1=e_2,\,e_2\ast e_2=e_2+e_3,\quad e_3\ast e_1=e_3,\quad 
	e_3\ast e_2=e_3,\, \alpha(e_1)=e_1,\\ \alpha(e_2)=e_2,\,\alpha(e_3)=e_3,\,\beta(e_1)=e_1,\,\beta(e_2)=e_1;$
	\item [$\mathcal{H}u^3_{8}$] : $e_1\ast e_1=e_1,\,e_1\ast e_2=e_1-e_2,\,e_2\ast e_1=e_2,\,e_3\ast e_1=e_3,\,e_3\ast e_2=e_3,\,\alpha(e_1)=e_1,\,\alpha(e_2)=e_2,\,\alpha(e_3)=e_3,\\ \beta(e_1)=e_1,\,\beta(e_2)=e_1-e_2;$
	\item [$\mathcal{H}^3_{9}$] : $e_1\ast e_1=e_1,\,e_1\ast e_2=e_1,\,e_2\ast e_1=e_2,\,e_2\ast e_2=e_2,\,e_3\ast e_1=e_3,\, e_3\ast e_2=e_3,\,
	\alpha(e_1)=e_1,\,\alpha(e_2)=e_2,\\ \alpha(e_3)=e_3,\, \beta(e_1)=e_1,\,\beta(e_2)=e_1;$
	\item [$\mathcal{H}u^3_{10}$] : $e_1\ast e_1=e_1,\,e_1\ast e_2=e_1,\,e_1\ast e_3=e_3,\,e_2\ast e_1=e_2,\,e_2\ast e_2=e_2,\, e_2\ast e_3=e_3,\,
	\alpha(e_1)=e_1,\,\alpha(e_2)=e_2,\\ \beta(e_1)=e_1,\,\beta(e_2)=e_1,\, \beta(e_3)=e_3;$
	\item [$\mathcal{H}u^3_{11}$] : $e_1\ast e_1=e_1,\quad e_1\ast e_3=e_3,\quad e_2\ast e_2=e_2,\quad e_3\ast e_2=e_2,\quad \quad\alpha(e_1)=e_1,\quad\beta(e_1)=e_1,\quad \beta(e_3)=e_3;$
	\item [$\mathcal{H}u^3_{12}$] : $e_1\ast e_1=e_1,\quad e_1\ast e_2=e_2,\quad e_2\ast e_3=e_3,\quad e_3\ast e_3=e_3,\quad\alpha(e_1)=e_1,\quad \beta(e_1)=e_1,\quad \beta(e_2)=e_2;$
		\item [$\mathcal{H}u^3_{13}$] : $e_1\ast e_1=e_1,\quad e_2\ast e_2=e_2,\quad e_2\ast e_3=e_2,\quad e_3\ast e_1=e_3,\quad\alpha(e_1)=e_1,\quad \alpha(e_3)=e_3,\quad \beta(e_1)=e_1.$
\end{itemize}

\section{BiHom-coassociative coalgebras and BiHom-bialgebras}
In this Section, we show that for a fixed dimension $n$, the set of BiHom-bialgebras is endowed with a structure of an algebraic variety and a natural structure transport action which describes the set of isomorphic BiHom-algebras. Solving such systems of polynomial equations leads
to the classification of such structures. We shall now introduce the dual concept to  BiHom-associative algebras:
\begin{definition}
A BiHom-coassociative coalgebra is $4$-tuple $(A, \Delta, \psi, \omega)$ in which $A$ is a linear space, $\psi, \omega : A\longrightarrow A$
and $\Delta : A\longrightarrow A\otimes A$ are linear maps, such that 
\begin{equation}
\psi\circ\omega=\omega\circ\psi, \quad (\psi\otimes\psi)\circ\Delta=\Delta\circ\psi, \quad (\omega\otimes\omega)\circ\Delta=\Delta\circ\omega,\quad 
(\Delta\otimes\psi)\circ\Delta=(\omega\otimes\Delta)\circ\Delta.
\end{equation}
We call $\psi$ and $\omega$ (in this order) the structure maps of $A$.

Let $V$ be an $n$-dimensional vector space over $\mathbb{K}$. Fixing a basis $\left\{e_i\right\}_{i=\left\{1,\dots,n\right\}}$ of $V$,
a multiplication $\mu$ (resp. linear maps $\alpha$, $\beta$, $\psi$, $\omega$ and a comultiplication $\Delta$) is identified with its $n^3$,
$2n^2$ structure constants $\mathcal{C}^k_{ij}\in\mathbb{K}$ (resp. $a_{ji}$, $b_{ji}$, $\xi_{ji}$, $D_i^{jk}$ and $\gamma_{ji})$
where $\mu(e_i\otimes e_j)=\sum_{k=1}^n\mathcal{C}_{ij}^ke_k$, $\alpha(e_i)=\sum_{j=1}^na_{ji}e_j$, $\beta(e_i)=\sum_{j=1}^nb_{ji}e_j,$
$\psi(e_i)=\sum_{j=1}^n\xi_{ji}e_j$, $\omega(e_i)=\sum^{n}_{j=1}\gamma_{ji}e_j$ and $\Delta(e_i)=\sum_{j,k=1}^nD_{i}^{jk}e_j\otimes e_k.$
The counit $\varepsilon$ is identified with its $n$ structure constants $\zeta_i.$ We assume that $e_1$ is the unit. 

A family $\left\{(\mathcal{C}_{ij}^k, a_{ji}, b_{ji},\xi_{ji},\gamma_{ji}, D_i^{jk}),\dots,i, j, k\in \left\{1,\dots,n\right\}\right\}$
represents a BiHom-coassociative coalgebra if the underlying family satisfies the appropriate conditions which translate to the following
polynomial equations:

\begin{equation}\label{s1}
\left\{\begin{array}{c}  
\sum_{k=1}^n\gamma_{ki}\xi_{pk}-\sum_{j=1}^n\xi_{ji}\gamma_{pj}=0,\quad \forall\,i,p\in\left\{1,\dots,n\right\},\\
\sum_{j=1}^n\sum_{k=1}^nD_i^{jk}\xi_{pj}\xi_{qk}-\sum_{j=1}^n\xi_{ji}D_j^{qp}=0,\quad \forall\, i,p,q\in\left\{1,\dots,n\right\},\\
\sum_{j=1}^n\sum_{k=1}^nD^{jk}_i\gamma_{pj}\gamma_{qk}-\sum_{k=1}^n\gamma_{ki}D^{pq}_k=0,\quad\forall\, i, p,q\in\left\{1,\dots,n\right\},\\
\sum_{j=1}^n\sum_{k=1}^n(D_i^{jk}\xi_{rk}-D_i^{jk}\gamma_{rj}D_k^{pq})\quad=0,\quad\forall\, i,p,q,r\in\left\{1,\dots,n\right\}.
\end{array}\right.
\end{equation}
A morphism $f : (A, \Delta_A, \psi_A, \omega_A)\longrightarrow (B, \Delta_B, \psi_B, \omega_B)$  of BiHom-coassociative coalgebras is a linear map 
$f : A\longrightarrow B$ such that $\psi_B\circ f=f\circ\psi_A,\quad\omega_B\circ f=f\circ\omega_A$ and $(f\otimes f)\circ \Delta_A=\Delta_B\circ f.$
\begin{equation}
\left\{\begin{array}{c}  
\sum_{j=1}^n(d_{ji}\xi_{kj}-\xi_{ji}d_{kj})=0,\quad \forall\,i,k\in\left\{1,\dots,n\right\},\\
\sum_{j=1}^n(d_{ji}\gamma_{kj}-\gamma_{ji}d_{kj})=0,\quad \forall\,i,k\in\left\{1,\dots,n\right\},\\
\sum_{j=1}^n\sum_{k=1}^nD_i^{jk}d_{pj}d_{qk}-\sum_{j=1}^nd_{ji}D_j^{qp}=0,\quad \forall\, i,p,q\in\left\{1,\dots,n\right\}.
\end{array}\right.
\end{equation}
A  BiHom-coassociative coalgebra $(A, \Delta, \psi, \omega)$ is called counital if there exists a linear map $\varepsilon : A\rightarrow \mathbb{K}$
(called a counit) such that 
\begin{equation}
\varepsilon\circ\psi=\varepsilon, \quad  \varepsilon\circ\omega=\varepsilon, \quad (id_A\otimes\varepsilon)\circ\Delta=\omega\quad\text{and}
\quad (\varepsilon\otimes id_A)\circ\Delta=\psi.
\end{equation}
We have $\sum_{j=1}^{n}\xi_{ji}\zeta_j=\zeta_i,\quad \sum_{j=1}^{n}\gamma_{ji}\zeta_j=\zeta_i,\quad
\sum_{p=1}^{n}\sum_{q=1}^{n}D_i^{pq}\varepsilon_q=a_{ij},\quad\sum_{p=1}^{n}\sum_{q=1}^{n}D_i^{qr}\varepsilon_r=b_{ij}.$

A morphism of counital BiHom-coassociative coalgebras $f : A\longrightarrow B$ is called counital if $f=\varepsilon_A$, where $\varepsilon_A$ and 
$\varepsilon_B$ are the counits of $A$ and $B$, respectively.
\end{definition}

\begin{definition}
A  BiHom-bialgebra is a $7$-tuple $(H,\mu,\Delta,\alpha,\beta,\psi,\omega)$, with the property that $(H,\mu, \alpha, \beta)$ is a BiHom-associative 
algebra, $(H,\Delta, \psi, \omega)$ is a BiHom-coassociative coalgebra and moreover the following relations are satisfied, for $h, h'\in H$ : 

\begin{equation}
\begin{gathered}
\Delta(hh')=h_1h'_1\otimes h_2h'_2, \quad \alpha\circ\psi=\psi\circ\alpha, \quad \alpha\circ\omega=\omega\circ\alpha,\\
\beta\circ\psi=\psi\circ\beta,\quad \beta\circ\omega=\omega\circ\beta,\quad (\alpha\otimes\alpha)\circ\Delta=\Delta\circ\alpha,\\
(\beta\otimes\beta)\circ\Delta=\Delta\circ\beta,\quad \psi(hh')=\psi(h)\psi(h'), \quad \omega(hh')=\omega(h)\omega(h').
\end{gathered}
\end{equation}

\begin{equation}\label{s1}
\left\{\begin{array}{c}  
\sum_{k=1}^n\mathcal{C}_{i,j}^kD^{pq}_k-\sum_{r=1}^n\sum_{s=1}^n\sum_{u=1}^n\sum_{v=1}^nD_i^{rs}D_j^{uv}\mathcal{C}^p_{ru}\mathcal{C}^q_{sv}=0,\quad \forall\,i,j,p,q\in\left\{1,\dots,n\right\},\\
\sum_{j=1}^n(\xi_{ji}a_{kj}-a_{ji}\xi_{kj})=0,\quad \forall\,i,k\in\left\{1,\dots,n\right\},\\
\sum_{j=1}^n(\gamma_{ji}a_{kj}-a_{ji}\gamma_{kj})=0,\quad \forall\,i,k\in\left\{1,\dots,n\right\},\\
\sum_{j=1}^n(\xi_{ji}b_{kj}-b_{ji}\xi_{kj})=0,\quad \forall\,i,k\in\left\{1,\dots,n\right\},\\
\sum_{j=1}^n(\gamma_{ji}b_{kj}-b_{ji}\gamma_{kj})=0,\quad \forall\,i,k\in\left\{1,\dots,n\right\},\\
\sum_{p=1}^n\sum_{q=1}^nD_i^{pq}a_{rp}a_{sq}-\sum_{j=1}^na_{ji}D_j^{rs}=0,\quad \forall\, i,r,s\in\left\{1,\dots,n\right\},\\
\sum_{p=1}^n\sum_{q=1}^nD^{pq}_ib_{rp}b_{sq}-\sum_{j=1}^nb_{ji}D^{rs}_j=0,\quad\forall\, i, r,s\in\left\{1,\dots,n\right\},\\
\sum_{k=1}^n\mathcal{C}^k_{ij}\xi_{pk}-\sum_{k=1}^n\sum_{q=1}^n\xi_{ki}\xi_{qj}\mathcal{C}^p_{kq}\quad=0,\quad\forall\,i,j,p\in\left\{1,\dots,n\right\},\\
\sum_{k=1}^n\mathcal{C}^k_{ij}\gamma_{pk}-\sum_{k=1}^n\sum_{q=1}^n\gamma_{ki}\gamma_{qj}\mathcal{C}^p_{kq}\quad=0,\quad\forall\,i,j,p \in\left\{1,\dots,n\right\}.
\end{array}\right.
\end{equation}

We say that $H$ is a unital and counital BiHom-bialgebra if, in addition, it admits a unit $u_H$ and a counit $\varepsilon_H$  such that
\begin{equation}
\begin{gathered}
\Delta(u_H)=u_H\otimes u_H, \quad \varepsilon_H(u_H)=1, \quad \psi(u_H)=u_H,\quad \omega(u_H)=u_H,\\
\varepsilon_H\circ\alpha=\varepsilon_H,\quad \varepsilon_H\circ\beta=\varepsilon_H,\quad \varepsilon_H(hh')=\varepsilon_H(h)\varepsilon_H(h'),
\quad \forall h,h'\in H.
\end{gathered}
\end{equation}
We have $\Delta(e_1)=e_1\otimes e_1,\quad\varepsilon(e_1)=1,\quad\psi(e_1)=e_1,\quad\omega(e_1)=e_1,
\quad \sum_{j=1}^{n}a_{ji}\zeta_j=\zeta_i,\quad\text{and}\quad\sum_{j=1}^{n}b_{ji}\eta_j=\eta_i.$

Let us record the formula expressing the BiHom-coassociativity of $\Delta$ : 
\begin{equation}
\Delta(h_1)\otimes\psi(h_2)=\omega(h_1)\otimes\Delta(h_2), \quad h\in H.
\end{equation}
\end{definition}

\subsection{Classification in Dimension 2 in $\mathcal{H}_i^2$}\ \\
Thanks to computer Hom-algebra, we obtain the following Hom-coalgebras associated to the previons Hom algebras in order to obtain a 
Hom-bialgebra structures. We denote the comultiplications by $\Delta^2_{i,j}$, where $i$ indicates the item of the multiplication and $j$
the item of comultiplication.
\begin{theorem}
The set of $2$-dimensional BiHom-Bialgebras algebras yields two non-isomorphic algebras. Let $\left\{e_1,e_2\right\}$ be a basis of 
$\mathbb{K}^2$, then the BiHom-Bialgebras are given by the following non-trivial comultiplications. 
\end{theorem}
\begin{enumerate}
\item [$1.$]$\Delta(e_1)=e_1\otimes e_2+e_2\otimes e_2,\,\Delta(e_2)=e_1\otimes e_1-e_2\otimes e_2,\,\psi(e_1)=-e_1,\,\psi(e_2)=e_2,\,
\omega(e_1)=-e_1,\,\omega(e_2)=e_2;$
\item[$2.$]$\Delta(e_1)=e_1\otimes e_1,\,\Delta(e_2)=e_1\otimes e_1+e_1\otimes e_2+e_2\otimes e_1+e_2\otimes e_2,\,\psi(e_1)=e_1,\,
\omega(e_1)=e_1,\,\omega(e_2)=e_2;$
\item[$3.$]$\Delta(e_1)=e_1\otimes e_1-e_1\otimes e_2-e_2\otimes e_1+e_2\otimes e_2,\,\Delta(e_2)=e_1\otimes e_1-e_1\otimes e_2-e_2\otimes e_1+e_2\otimes e_2,\,\quad\psi(e_1)=e_1,\,\psi(e_2)=e_2,\,\omega(e_1)=e_1-e_2,\,\omega(e_2)=e_1-e_2;$
\item[$4.$]$\Delta(e_1)=e_1\otimes e_1-e_1\otimes e_2-e_2\otimes e_1+e_2\otimes e_2,\,\Delta(e_2)=e_1\otimes e_1-e_1\otimes e_2-e_2\otimes e_1+e_2\otimes e_2,\,\psi(e_1)=e_1-e_2,\,\psi(e_2)=e_1-e_2,\,\omega(e_1)=e_1,\,\omega(e_2)=e_2;$
\item[$5.$]$\Delta(e_1)=e_1\otimes e_1-e_1\otimes e_2-e_2\otimes e_1+e_2\otimes e_2,\,\Delta(e_2)=e_1\otimes e_1-e_1\otimes e_2-e_2\otimes e_1+e_2\otimes e_2,\,\psi(e_1)=e_1-e_2,\,\psi(e_2)=e_1-e_2,\,\omega(e_1)=e_1-e_2,\,\omega(e_2)=e_1-e_2;$
\item[$6.$]$\Delta(e_1)=e_1\otimes e_1-e_1\otimes e_2-e_2\otimes e_1+e_2\otimes e_2,\,\Delta(e_2)=e_1\otimes e_1-e_1\otimes e_2-e_2\otimes e_1+e_2\otimes e_2,\,\psi(e_1)=-e_1+e_2,\,\psi(e_2)=-e_1+e_2,\,\omega(e_1)=e_1,\,\omega(e_2)=e_2;$
\item[$7.$]$\Delta(e_1)=e_1\otimes e_1-e_1\otimes e_2-e_2\otimes e_1+e_2\otimes e_2,\,\Delta(e_2)=e_1\otimes e_1-e_1\otimes e_2-e_2\otimes e_1+e_2\otimes e_2,\,\psi(e_1)=e_1,\,\psi(e_2)=e_2;$
\item[$8.$]$\Delta(e_1)=e_1\otimes e_1+e_1\otimes e_2+e_2\otimes e_1+e_2\otimes e_2,\,\Delta(e_2)=e_1\otimes e_1+e_1\otimes e_2+e_2\otimes e_1+e_2\otimes e_2,\,\psi(e_1)=e_1-e_2,\,\psi(e_2)=-e_1+e_2;$
	\item[$9.$]$\Delta(e_1)=e_1\otimes e_1+e_1\otimes e_2+e_2\otimes e_1+e_2\otimes e_2,\,\Delta(e_2)=e_1\otimes e_1+e_1\otimes e_2+e_2\otimes e_1+e_2\otimes e_2,\,\psi(e_1)=e_1-e_2,\,\psi(e_2)=-e_1+e_2,\,\omega(e_1)=e_1-e_2,\,\omega(e_2)=-e_1+e_2;$
	\item[$10.$]$\Delta(e_1)=e_1\otimes e_1+e_1\otimes e_2+e_2\otimes e_1+e_2\otimes e_2,\,\Delta(e_2)=e_1\otimes e_1+e_1\otimes e_2+e_2\otimes e_1+e_2\otimes e_2,\,\omega(e_1)=e_1+e_2,\,\omega(e_2)=e_1+e_2.$
\end{enumerate}
\begin{remark}
There is no BiHom-Bialgabra whose underlying BiHom-associative algebra is given by $\mathcal{H}_{1}^2$.
\end{remark}
\begin{theorem}
The set of $2$-dimensional unital BiHom-Bialgebras yields two non-isomorphic algebras. Let $\left\{e_1,e_2\right\}$ be a basis of 
$\mathbb{K}^2$, then the unital BiHom-Bialgebras are given by the following non-trivial comultiplications.  
\end{theorem}
\begin{enumerate}
\item[$1.$]$\Delta^2_{1,1}(e_1)=e_1\otimes e_1,\,\Delta^2_{1,1}(e_2)=-e_1\otimes e_1+e_1\otimes e_2+e_2\otimes e_1+e_2\otimes e_2,\,\psi(e_1)=e_1,\,\psi(e_2)=e_2,\,\omega(e_1)=e_1,\\ \omega(e_2)=e_2,\,\varepsilon(e_1)=1,\,\varepsilon(e_2)=2;$
\item[$2.$]$\Delta^2_{1,2}(e_1)=e_1\otimes e_1,\,\Delta^2_{1,2}(e_2)=-e_1\otimes e_1+e_1\otimes e_2+e_2\otimes e_1+e_2\otimes e_2,\,\psi(e_1)=e_1,\,\psi(e_2)=e_1,\,\omega(e_1)=e_1,\\ \omega(e_2)=e_1,\,\varepsilon(e_1)=1,\,\varepsilon(e_2)=2;$
\item[$3.$]$\Delta^2_{1,3}(e_1)=e_1\otimes e_1,\,\Delta^2_{1,3}(e_2)=e_1\otimes e_1+e_1\otimes e_2+e_2\otimes e_1-e_2\otimes e_2,\,\psi(e_1)=e_1,\,\psi(e_2)=e_1,\,\omega(e_1)=e_1,\\ \omega(e_2)=e_1,\,\varepsilon(e_1)=1,\,\varepsilon(e_2)=-1;$
\item[$4.$]$\Delta^2_{1,4}(e_1)=e_1\otimes e_1,\,\Delta^2_{1,4}(e_2)=e_1\otimes e_1+e_1\otimes e_2+e_2\otimes e_1-e_2\otimes e_2,\,\psi(e_1)=e_1,\,\psi(e_2)=-e_1,\,\omega(e_1)=e_1,\\ \omega(e_2)=-e_1,\,\varepsilon(e_1)=1,\,\varepsilon(e_2)=-1;$
\item[$5.$]$\Delta^2_{2,1}(e_1)=e_1\otimes e_1,\,\Delta^2_{2,2}(e_2)=-e_1\otimes e_1+e_1\otimes e_2+e_2\otimes e_1+e_2\otimes e_2,\,
\psi(e_1)=e_1,\,\psi(e_2)=e_1,\,\omega(e_1)=e_1,\\ \omega(e_2)=e_1,\,\varepsilon(e_1)=1,\,\varepsilon(e_2)=1;$
\item[$6.$]$\Delta^2_{2,2}(e_1)=e_1\otimes e_1,\,\Delta^2_{2,3}(e_2)=e_1\otimes e_1+e_1\otimes e_2+e_2\otimes e_1-e_2\otimes e_2,\,
\psi(e_1)=e_1,\,\psi(e_2)=e_2,\,\omega(e_1)=e_1,\\ \omega(e_2)=e_2,\,\varepsilon(e_1)=1,\,\varepsilon(e_2)=1;$
\item[$7.$]$\Delta^2_{4,1}(e_1)=e_1\otimes e_1,\,\Delta^2_{3,1}(e_2)=e_2\otimes e_2,\,\psi(e_1)=e_1,\,\psi(e_2)=e_2,\,\omega(e_1)=e_1,\,
\omega(e_2)=e_2,\,\varepsilon(e_1)=1,\,\varepsilon(e_2)=1;$
\item[$8.$]$\Delta^2_{4,2}(e_1)=e_1\otimes e_1,\,\Delta^2_{4,2}(e_2)=e_2\otimes e_2,\,\psi(e_1)=e_1,\,\psi(e_2)=e_1,\,\omega(e_1)=e_1,\,
\omega(e_2)=e_1,\,\varepsilon(e_1)=1,\,\varepsilon(e_2)=1;$
\item[$9.$]$\Delta^2_{4,3}(e_1)=e_1\otimes e_1,\,\Delta^2_{4,3}(e_2)=e_1\otimes e_2+e_2\otimes e_1-2e_2\otimes e_2,\,\psi(e_1)=e_1,\,\psi(e_2)=e_2,\,\omega(e_1)=e_1,\,\omega(e_2)=e_2,\\ \varepsilon(e_1)=1,\,\varepsilon(e_2)=1;$
\item[$10.$]$\Delta^2_{4,4}(e_1)=e_1\otimes e_1,\,\Delta^2_{4,4}(e_2)=e_1\otimes e_2+e_2\otimes e_1-e_2\otimes e_2,\,\psi(e_1)=e_1,\,\psi(e_2)=e_1,\,\omega(e_1)=e_1,\,\omega(e_2)=e_1,\\ \varepsilon(e_1)=1;$
\end{enumerate}
\begin{remark}
There is no BiHom-Bialgabra whose underlying BiHom-associative algebra is given by $\mathcal{H}u_3^2.$ 
\end{remark}

\subsection{Classification in Dimension 3 in $\mathcal{H}_i^3$}\ \\
Thanks to computer Hom-algebra, we obtain the following Hom-coalgebras associated to the previons Hom algebras in order to obtain a 
Hom-bialgebra structures. We denote the comultiplications by $\Delta^3_{i,j}$, where $i$ indicates the item of the multiplication and $j$
the item of comultiplication.
\begin{theorem}
The set of $3$-dimensional BiHom-Bialgebras yields two non-isomorphic algebras. Let $\left\{e_1,e_2,e_3\right\}$ be a basis of 
$\mathbb{K}^3$, then the BiHom-Bialgebras are given by the following non-trivial comultiplications.  
\end{theorem}
\begin{enumerate}
\item[$1.$]$\Delta(e_1)=e_1\otimes e_1,\,\Delta(e_2)=e_1\otimes e_1+e_1\otimes e_2-e_1\otimes e_3+e_2\otimes e_1+e_2\otimes e_3-e_3\otimes e_1-e_3\otimes e_3,$\\ $\Delta(e_3)=e_3\otimes e_3,\,\psi(e_1)=e_1,\,\psi(e_2)=e_1+e_2,\,\psi(e_3)=e_3,\,\omega(e_1)=e_1,\,\omega(e_2)=e_1+e_2,\,
\omega(e_3)=e_3;$
\item[$2.$]$\Delta(e_1)=e_1\otimes e_1,\,\Delta(e_2)=e_1\otimes e_1+e_1\otimes e_2+e_1\otimes e_3+e_2\otimes e_1-e_2\otimes e_2+e_2\otimes e_3-e_3\otimes e_1+e_3\otimes e_2$\\ $\Delta(e_3)=e_3\otimes e_3,\,\psi(e_1)=e_1,\,\psi(e_2)=e_1,\,\psi(e_3)=e_3,\,\omega(e_1)=e_1,\,\omega(e_2)=e_1,\,\omega(e_3)=e_3;$
\item[$3.$]$\Delta(e_1)=e_1\otimes e_1+e_3\otimes e_3,\quad\Delta(e_2)=e_1\otimes e_2+e_2\otimes e_1,\quad\Delta(e_3)=e_1\otimes e_3+e_3\otimes e_1,\quad \psi(e_1)=e_1,\\ \psi(e_3)=e_3,\quad\omega(e_1)=e_1,\quad\omega(e_3)=e_3;$
\item[$4.$]$\Delta(e_1)=e_1\otimes e_1+e_3\otimes e_3,\quad\Delta(e_2)=e_1\otimes e_2+e_2\otimes e_1,\quad\Delta(e_3)=-e_1\otimes e_3-e_3\otimes e_1,\quad \psi(e_1)=e_1,\\ \psi(e_3)=e_3,\quad \omega(e_1)=e_1,\quad\omega(e_3)=e_3;$
\item[$5.$]$\Delta(e_1)=e_1\otimes e_1,\,\,\Delta(e_2)=e_1\otimes e_2+e_2\otimes e_1+e_2 \otimes e_2,\,\,\Delta(e_3)=e_1\otimes e_3+
e_2\otimes e_3+e_3\otimes e_1+e_3\otimes e_2-e_3\otimes e_3,\\ \psi(e_1)=e_1,\quad \omega(e_1)=e_1;$
\item[$6.$]$\Delta(e_1)=e_1\otimes e_1,\,\,\Delta(e_2)=e_1\otimes e_2+e_2\otimes e_1+e_2 \otimes e_2,\,\,\Delta(e_3)=e_1\otimes e_3+
e_2\otimes e_3+e_3\otimes e_1+e_3\otimes e_2-e_3\otimes e_3,\\ \psi(e_1)=e_1,\quad\psi(e_2)=e_2,\,\omega(e_1)=e_1,\quad \omega(e_2)=e_2;$
\item[$7.$]$\Delta(e_1)=0,\quad\Delta(e_2)=-e_1\otimes e_1+e_1\otimes e_2+e_2 \otimes e_1,\quad\Delta(e_3)=e_1\otimes e_3+e_3\otimes e_1,\, \psi(e_1)=e_1,\,\,\psi(e_2)=e_1-e_2,\,\psi(e_3)=-e_3,\,\,\omega(e_1)=e_1,\quad\omega(e_2)=e_1-e_2,\,\,\omega(e_3)=-e_3;$
\item[$8.$]$\Delta(e_1)=0,\quad\Delta(e_2)=-e_1\otimes e_1+e_1\otimes e_2+e_2 \otimes e_1,\quad\Delta(e_3)=e_1\otimes e_3+e_3\otimes e_1,\, \psi(e_1)=e_1,\,\,\psi(e_2)=e_2,\,\psi(e_3)=e_3,\,\,\omega(e_1)=e_1,\quad\omega(e_2)=e_1-e_2,\,\, \omega(e_3)=-e_3;$
\item[$9.$]$\Delta(e_1)=0,\quad\Delta(e_2)=-e_1\otimes e_1+e_1\otimes e_2+e_2 \otimes e_1,\quad\Delta(e_3)=e_1\otimes e_3+e_3\otimes e_1,\, \psi(e_1)=e_1,\,\,\psi(e_2)=e_1-e_2,\,\psi(e_3)=-e_3,\,\,\omega(e_1)=e_1,\quad\omega(e_2)=e_2,\,\,\omega(e_3)=e_3.$
\end{enumerate}

\begin{remark}
There is no BiHom-Bialgabra whose underlying BiHom-associative algebra is given by $\mathcal{H}_3^3,\,\mathcal{H}_5^3.$
\end{remark}

\begin{theorem}
The set of $3$-dimensional unital BiHom-Bialgebras yields two non-isomorphic algebras. Let $\left\{e_1,e_2,e_3\right\}$ be a basis of 
$\mathbb{K}^3$, then the unital BiHom-Bialgebras are given by the following non-trivial comultiplications.  
\end{theorem}
\begin{enumerate}
\item[$1.$]$\Delta^3_{2,1}(e_1)=e_1\otimes e_1,\,\Delta(e_2)=-e_1\otimes e_1+e_1\otimes e_2+e_2\otimes e_1,\,
\Delta(e_3)=e_1\otimes e_1-e_1\otimes e_2+2e_1\otimes e_3-e_2\otimes e_1+e_2\otimes e_2-e_2\otimes e_2+2e_3\otimes e_1-e_3\otimes e_2+e_3\otimes e_3,\,\psi(e_1)=e_1,\,\psi(e_2)=e_1,\,\omega(e_1)=e_1,\,\omega(e_2)=e_1,\,\varepsilon(e_1)=\varepsilon(e_2)=1;$
\item[$2.$]$\Delta^3_{3,1}(e_1)=e_1\otimes e_1,\,\Delta(e_2)=e_1\otimes e_1+e_1\otimes e_2+e_1\otimes e_3+e_2\otimes e_1-e_2\otimes e_2-e_2\otimes e_3-e_3\otimes e_1+e_3\otimes e_2,\,\Delta(e_3)=-ae_3\otimes e_3,\,\psi(e_1)=e_1,\,\psi(e_2)=e_2,\,\psi(e_3)=e_3,\,\omega(e_1)=e_1,\,\omega(e_2)=e_2,\,\omega(e_3)=e_3,\,\varepsilon(e_1)=\varepsilon(e_2)=1;$
\item[$3.$]$\Delta^3_{4,1}(e_1)=e_1\otimes e_1,\,\Delta(e_2)=e_2\otimes e_2,\,\Delta(e_3)=e_1\otimes e_3+e_3\otimes e_1+e_3\otimes e_3,\,\psi(e_1)=e_1,\,\psi(e_2)=e_2,\,\psi(e_3)=e_3,\,\omega(e_1)=e_1,\,\omega(e_2)=e_2,\,\omega(e_3)=e_3,\,\varepsilon(e_1)=\varepsilon(e_2)=1;$
\item[$4.$]$\Delta^3_{12,1}(e_1)=e_1\otimes e_1,\,\Delta(e_2)=e_1\otimes e_2+e_2\otimes e_1,\,\Delta(e_3)=e_1\otimes e_3+e_3\otimes e_1,\,
\psi(e_1)=e_1,\,\omega(e_1)=e_1,\,\varepsilon(e_1)=1;$
\item[$5.$]$\Delta^3_{15,1}(e_1)=e_1\otimes e_1,\,\Delta(e_2)=e_1\otimes e_1+e_1\otimes e_2-e_1\otimes e_3+e_2\otimes e_1-e_2\otimes e_3-e_3\otimes e_1-e_3\otimes e_2+2e_3\otimes e_3,\,\Delta(e_3)=-e_2\otimes e_2+e_2\otimes e_3+e_3\otimes e_2,\,\psi(e_1)=e_1,\,\psi(e_3)=e_1,\,
\omega(e_1)=e_1,\,\omega(e_3)=e_1,\,\varepsilon(e_1)=\varepsilon(e_3)=1;$
\item[$6.$]$\Delta^3_{15,2}(e_1)=e_1\otimes e_1,\,\Delta(e_2)=e_1\otimes e_2+e_2\otimes e_1-ae_2\otimes e_2-e_2\otimes e_3-e_3\otimes e_2,\,
\Delta(e_3)=e_1\otimes e_3+be_2\otimes e_2-e_3\otimes e_1-e_3\otimes e_3,\,\psi(e_1)=e_1,\,\psi(e_2)=ce_2,\,\psi(e_3)=de_2+e_3,\,
\omega(e_1)=e_1,\,\omega(e_2)=ce_2,\,\omega(e_3)=de_2+e_3,\,\varepsilon(e_1)=1;$
\item[$7.$]$\Delta^3_{20,1}(e_1)=e_1\otimes e_1,\,\Delta(e_2)=e_2\otimes e_2,\,\Delta(e_3)=e_1\otimes e_3+e_3\otimes e_1+e_3\otimes e_3,\,\psi(e_1)=e_1,\,,\,\psi(e_2)=e_2,\,\psi(e_3)=e_3,\,\omega(e_1)=e_1,\,\omega(e_2)=e_2,\,\omega(e_3)=e_3,\,\varepsilon(e_1)=\varepsilon(e_2)=1;$
\item[$8.$]$\Delta^3_{21,1}(e_1)=e_1\otimes e_1,\,\Delta(e_2)=e_1\otimes e_2+e_2\otimes e_1+e_2\otimes e_2,\,\Delta(e_3)=e_3\otimes e_3,\,\psi(e_1)=e_1,\,\,\psi(e_2)=e_2,\,\omega(e_1)=e_1,\,\omega(e_2)=e_2,\,\varepsilon(e_1)=1,\,\varepsilon(e_2)=1;$
\item[$9.$]$\Delta^3_{21,2}(e_1)=e_1\otimes e_1,\,\Delta(e_2)=e_1\otimes e_2+e_2\otimes e_1,\,\Delta(e_3)=e_1+\otimes e_3+e_3\otimes e_1,\,\psi(e_1)=e_1,\,\omega(e_1)=e_1,\,\varepsilon(e_1)=1;$
\item[$10.$]$\Delta^3_{22,1}(e_1)=e_1\otimes e_1,\,\Delta(e_2)=e_1\otimes e_2+e_2\otimes e_1+e_2\otimes e_2+e_2\otimes e_3+e_3\otimes e_2,\,
\Delta(e_3)=e_1+\otimes e_3+e_3\otimes e_1+e_3\otimes e_3,\,\psi(e_1)=e_1,\,\psi(e_3)=e_3,\,\omega(e_1)=e_1,\,\omega(e_3)=e_3,\,
\varepsilon(e_1)=1;$
\item[$11.$]$\Delta^3_{22,2}(e_1)=e_1\otimes e_1,\,\Delta(e_2)=e_2\otimes e_2,\,\Delta(e_3)=e_1+\otimes e_3+e_3\otimes e_1+e_3\otimes e_3,\,\psi(e_1)=e_1,\,\omega(e_1)=e_1,\,\varepsilon(e_1)=1;$
\item[$12.$]$\Delta^3_{22,3}(e_1)=e_1\otimes e_1,\,\Delta(e_2)=e_2\otimes e_2,\,\Delta(e_3)=e_1+\otimes e_3+e_3\otimes e_1+e_3\otimes e_3,\,
\psi(e_1)=e_1,\,\psi(e_3)=e_3,\,\omega(e_1)=e_1,\,\omega(e_3)=e_3,\,\varepsilon(e_1)=1;$
\item[$13.$]$\Delta^3_{22,4}(e_1)=e_1\otimes e_1,\,\Delta(e_2)=e_2\otimes e_2,\,\Delta(e_3)=e_1+\otimes e_3+e_3\otimes e_1+e_3\otimes e_3,\,
\psi(e_1)=e_1,\,\psi(e_2)=e_2,\,\psi(e_3)=e_3,\,\omega(e_1)=e_1,\,\omega(e_2)=e_2,\,\omega(e_3)=e_3,\,\varepsilon(e_1)=\varepsilon(e_2)=1.$
\end{enumerate}
\begin{remark}
There is no unital BiHom-Bialgabra whose underlying unital BiHom-associative algebra is given by 
$\mathcal{H}u_1^3,\,\mathcal{H}u_5^3,\,\mathcal{H}u_7^3,\,\mathcal{H}u_{8}^3,\,\mathcal{H}u_{10}^3.$ 
\end{remark}
We introduce the concept of BiHom-Hopf algebras.
\begin{definition}\label{bh1}
Let $(H,\mu,\Delta,\alpha,\beta)$ be a unital and counital BiHom-bialgebra with a unit $1_H$ and a co-unit $\varepsilon_H.$ A linear map 
$S : H\rightarrow H$ is called an antipode if it commutes with all the maps $\alpha,\beta,\psi,\omega$ and it satisfies the relation :
$$
\psi\omega(S(h_1))\alpha\beta(h_2)=\varepsilon_H(h)1_H=\beta\psi(h_1)\alpha\omega(S(h_2)),\quad \forall h\in H.
$$ 
A BiHom-Hopf algebra is a unital and counital BiHom-bialgebra with an antipode.

$\sum_{s=1}^n\sum_{q,r=1}^n\sum_{l,p=1}^n\sum_{j,k=1}^nD^{jk}_iS_{lj}g_{pk}f_{ql}a_{rp}b_{qr}\mathcal{C}^t_{sr}-\xi_{i}=0,\quad \forall\, i,t\in \left\{1,n\right\};$

$\sum_{s=1}^n\sum_{q,r=1}^n\sum_{l,p=1}^n\sum_{j,k=1}^nD^{jk}_if_{lj}S_{lj}b_{ql}g_{rp}a_{sr}\mathcal{C}^t_{qr}-\xi_{i}=0,\quad \forall\, i,t\in \left\{1,n\right\}.$
\end{definition}

\begin{proposition}
The BiHom-bialgebra structures on $\mathbb{K}^2$ which are BiHom-Hopf algebras are given by the following pairs  of multiplication and comultiplication
with the appropriate unit and conits :\\
$(\mathcal{H}u_1^2, \Delta_{1,1}^2),\,(\mathcal{H}u_1^2, \Delta_{1,2}^2),\,(\mathcal{H}u_1^2, \Delta_{1,4}^2),\,(\mathcal{H}u_2^2, \Delta_{2,1}^2),\,(\mathcal{H}u_2^2, \Delta_{2,2}^2),\,(\mathcal{H}u_4^2, \Delta_{4,1}^2),\,(\mathcal{H}u_4^2, \Delta_{4,2}^2),\,(\mathcal{H}u_4^2, \Delta_{4,3}^2),\,(\mathcal{H}u_4^2, \Delta_{4,4}^2).$
\end{proposition}
\begin{proposition}
The BiHom-bialgebra structures on $\mathbb{K}^3$ which are BiHom-Hopf algebras are given by the following pairs  of multiplication and comultiplication
with the appropriate unit and conits :\\
$(\mathcal{H}u_3^3, \Delta_{3,1}^3),\,(\mathcal{H}u_4^3, \Delta_{4,1}^3),\,(\mathcal{H}u_{12}^3, \Delta_{12,1}^3),
(\mathcal{H}u_{15}^3, \Delta_{15,1}^3),\,(\mathcal{H}u_{15}^3, \Delta_{15,2}^3),\,(\mathcal{H}u_{21}^3, \Delta_{21,1}^3),\,
(\mathcal{H}u_{21}^3, \Delta_{21,2}^3),\,(\mathcal{H}u_{22}^3, \Delta_{22,1}^3)$,\\ $(\mathcal{H}u_{22}^3, \Delta_{22,2}^3),$
$(\mathcal{H}u_{22}^3, \Delta_{22,3}^3),\,(\mathcal{H}u_{22}^3, \Delta_{22,4}^3).$ 
\end{proposition}

\end{document}